\documentclass[12pt,leqno]{amsart}
\usepackage{amsmath, amssymb, latexsym, amscd, amsthm,amsfonts,amstext}
\usepackage[mathscr]{eucal}

\def\underset#1#2{{\mathrel{\mathop {{}_{} {#2}}\limits_{{#1}_{}}}}}
\def\upplim_#1{\underset{#1}{\overline\lim}\;}
\def\lowlim_#1{\underset{#1}{\underline\lim}\;}
\newtheorem{The}{Theorem}
\newtheorem{Lem}[The]{Lemma}

\newtheorem{Cor}[The]{Corollary}

\newtheorem{Def}[The]{Definition}

\newcommand{\C}{{\mathbf{C}}}

\renewcommand{\P}{{\mathbf{P}}}

\newcommand{\supp}{\mathrm{Supp}\,}

\numberwithin{equation}{section}

\begin{document} 
	\title[On holomorphic curves into complex projective varieties]{On holomorphic curves into complex projective varieties} 
	
	\author{Giang Le}
	
	\setlength{\baselineskip}{16pt}
	\maketitle
	
	\begin{center}
		{\it Department of Mathematics, Hanoi National University of Education,}
		
		{\it 136-Xuan Thuy, Cau Giay, Hanoi, Vietnam.}
		
		\textit{E-mail: legiang@hnue.edu.vn, legiang01@yahoo.com}
	\end{center}
	
	\begin{abstract} {In this paper, we study holomorphic curves satisfying the Fubini-Study derivative $\|f'(z)\|=O(|z|^\sigma)$ for some $\sigma>-1$ from the viewpoint of Nevanlinna theory.}
	\end{abstract}
	\def\thefootnote{\empty}
	\setlength{\baselineskip}{16pt}
	\maketitle
	\section{Introduction}
	In this paper, we study holomorphic curves $f: \C\rightarrow\P^N(\C)$ satisfying the Fubini-Study derivative $\|f'(z)\|=O(|z|^\sigma)$ for some $\sigma >-1.$ The explicit expression is: 
	$$\| f'\|^2=
	\dfrac{\sum_{i\neq j}|f_i^\prime f_j-f_if_j^\prime|^2}{\|  f\|^4},$$
	where ${\bf f}=(f_0,\ldots,f_N)$ is a homogeneous representation of $f$ (that is, the $f_j$ are entire functions which
	never simultaneously
	vanish) and
		$$\|  f\|^2=\sum_{j=0}^N|f_j|^2.$$

When $\sigma=0$, $f$ is called a Brody curve. The interest of these curves comes from Brody's lemma.  It is stated that for every non constant holomorphic curve $f$, there exists a sequence of affine map $a_k:\C\rightarrow\C$ such that the limit $\int f\circ a_k$ exists and is a non constant Brody curves. Much attention has been given to study Brody curves from the viewpoint of Nevanlinna theory \cite{BE, DD, T0, T1}.  In \cite{BE, DD, TM}, some versions of Second Main Theorem have been established for  Brody curves intersecting hypersurfaces located in general position. The first aim of this paper is to generalize such the results to the case of holomorphic curves satisfying  the Fubini-Study derivative $\|f'(z)\|=O(|z|^\sigma)$ for some $\sigma >-1$ and arbitrary hypersurfaces.

	

	We recall that the Nevanlinna characteristic is
defined as follows
$$T_f(r)=\int_0^r \frac{dt}{t}\left(\frac{1}{\pi}\int_{|z|\leq t}\| f'\|^2(z)dm_z\right),$$
where $dm$ is the area element in $\C$.

	Obviously, the condition $\|f'(z)\|=O(|z|^\sigma)$ for some $\sigma>-1$ implies that $$T_f(r)=O(r^{2\sigma+2}).$$
	Clunie and Hayman \cite{CH} found that if the spherical derivative $\|f'(z)\|$ of an entire function satisfies $\|f'(z)\|=O(|z|^\sigma)$, then $T_f(r)=O(r^{\sigma+1})$. Barret and Eremenko \cite{BE} generalized this to the case of holomorphic curves in projective space of dimension $n$ omitting $n$ hyperplanes in general position. The first main result of this paper is to show that this phenomenon persists in complex projective variety omitting hypersurfaces satisfying the intersection of them consisting of finite points.
	
	In this paper, we always assume that $V$ be a complex projective algebraic variety in $\P^N(\C)$.
	\begin{The}\label{t1}
	 Let $D_1,\ldots, D_q$ be $q$ hypersurfaces  such that $\dim(\bigcap_{1\leq i\leq q}\supp D_i\bigcap V)=0$. Let $f: \C\rightarrow V$ satisfying $\|f'(z)\|=O(|z|^\sigma)$ for some $\sigma>-1.$ Assume that $(\bigcup_{1\leq i\leq q}\supp D_i)\bigcap f(\C)=\emptyset.$ Then,
		$$T_f(r)=O(r^{\sigma+1}).$$
		\end{The}
	Combined with a result of Tsukamoto \cite[Theorem 1.9]{T0}, Theorem \ref{t1} implies that
	\begin{Cor}
	Mean dimension in the sense of Gromov \cite{Gromov} of the space of Brody curves in $$V\backslash \text{ $\dim V$ hypersurfaces in general position}$$ is zero.
	\end{Cor}
In \cite{DD}, Da Costa and Duval proved a second main theorem for Brody curves in projective space intersecting hyperplanes. Later, Thai and Mai \cite{TM} established a second main theorem over an angular domain for Brody curves into a complex projective algebraic variety  intersecting hypersurfaces in general position. Our second purpose is to generalize Thai-Mai's results \cite{TM} to the case of arbitrary hypersurfaces and holomorphic curves satisfying $\|f'(z)\|=O(|z|^\sigma)$ for some $\sigma >-1.$ To state our result, we recall some notations.

\begin{Def} Let $D_1,\ldots, D_q$ be hypersurfaces in $\P^N(\C)$.  

a)	They are said to be in $m-$subgeneral position in $V$ ($\dim V\leq m\in\mathbb{N}$) if for any subset $I \subset\{1,\ldots,q\}$ with $\sharp I\leq m+1$, we have $$ \dim\bigcap_{i\in I}\supp D_{i}\bigcap V\leq m-\sharp I.$$ 
	
b)	They  are said to be in general position in $V$ if they are in $\dim V$-subgeneral position.
\end{Def}
Let  $f: \C\rightarrow\P^N(\C)$ be a holomorphic curve. We choose a homogeneous representation ${\bf f}=(f_0,\ldots, f_N)$ of our curve. Set
$$u(z)=\log\sqrt{|f_0(z)|^2+\ldots+|f_N(z)|^2}.$$
Obviously, $u$ is a positive subharmonic function . 
Let $\triangle u$ be the Riesz measure of $u$, that is the measure with the density
$$\triangle u=\dfrac{1}{\pi}\|f'\|^2.$$
Let 
$$\mathbb{D}=\{z\in\C
: |z|< 1\}, r\mathbb{D}=\{z: |z|< r\}$$
throughout the paper. Let $\Omega$ be a domain in $\mathbb{C}.$ The Ahlfors-Shimizu characteristic function of $f$ is defined as follows
\begin{align*}
	T_f(r,\Omega)=\int_1^r\dfrac{\triangle u(\Omega\bigcap t\mathbb{D})}{t}dt+O(1).
\end{align*}
It is well-known that 
$$T_f(r,\mathbb{C})=T_f(r)=\dfrac{1}{2\pi}\int_0^{2\pi}u(re^{i\theta})d\theta+O(1).$$
Let $D$ be a hypersurface in $\P^N(\C)$ defining by a homogeneous polynomial $Q$.  
For a domain $\Omega$ denote by $n_f(r, \Omega, D)$ the number of zeros of $Q\circ {\bf f}$ in the  domain $\{|z|<r\}\bigcap\Omega,$ counting multiplicity. The counting function is defined as follows:
$$N_f(r,\Omega, D)=\int_{0}^{r}\dfrac{n_f(r, \Omega, D)-n_f(0, \Omega, D)}{t}dt+n_f(0,\Omega, D)\log r.$$
In this paper, we consider  angular domains $\Omega (\theta, \epsilon), 0\leq \theta\leq 2\pi, 0<\epsilon<\pi$ defined as follows
$$\Omega (\theta, \epsilon)=\{z: \theta-\epsilon<\arg z<\theta+\epsilon\}.$$
Sometimes, we write simply $\Omega$ instead of $\Omega(\theta, \epsilon)$ if there is no confusion.
\begin{The}\label{t2}
	 Let $f: \C\rightarrow V$ be a holomorphic curve satisfying $\|f'(z)\|=O(|z|^{\sigma})$ for some $\sigma>-1$. Let $D_1,\ldots, D_q$ be hypersurfaces of degrees $d_1,\ldots, d_q$ in $\P^N(\C)$. Let $\Omega=\Omega(\theta, \epsilon)$ be any angular domain in $\C.$
	
	(i) Let $\dim V\leq m\in \mathbb{N}.$ If  $D_1,\ldots, D_q$ are located in $m-$subgeneral position in $V$ then we have
	$$(q-m+1)T_f(r, \Omega)\leq \sum_{1\leq i\leq q}\dfrac{1}{d_i}N_f(r, \Omega, D_i)+o(r^{2\sigma+2}).$$
	
	(ii) If $V\bigcap\supp D_1\bigcap\ldots \bigcap\supp D_q=\emptyset$ then we have
	$$2T_f(r, \Omega)\leq \sum_{1\leq i\leq q}\dfrac{1}{d_i}N_f(r, \Omega, D_i)+o(r^{2\sigma+2}).$$
	
	iii) If $\dim(V\bigcap\supp D_1\bigcap\ldots \bigcap\supp D_q)=0$ then we have
	$$T_f(r, \Omega)\leq \sum_{1\leq i\leq q}\dfrac{1}{d_i}N_f(r, \Omega, D_i)+o(r^{2\sigma+2}).$$
\end{The}
Remark that, we obtain the part (ii) from the part (i) by observing that $D_1,\ldots, D_q$ satisfying the condition  $V\bigcap\supp D_1\ldots \bigcap\supp D_q=\emptyset$ could be considered to be located in $(q-1)-$subgeneral position. Similarly, the part (iii) comes from the part (i) by observing that $D_1,\ldots, D_q$ satisfying the condition  $\dim(V\bigcap\supp D_1\bigcap\ldots \bigcap\supp D_q)=0$ could be considered to be located in $q-$subgeneral position.

As a consequence of Theorem \ref{t2}, for holomorphic curve satisfying $\limsup\limits_{r\rightarrow\infty} \dfrac{T_f(r)}{r^{2\sigma+2}}>0$, we have the defect relation 
$$\sum_{1\leq i\leq q}\delta_{D_i}(f)\leq m-1$$
for a family of hypersurfaces $D_1,\ldots, D_q$ located in $m-$subgeneral position. For a general holomorphic curve, the best result available at present is due to Quang\cite{Q}, where the defect relation is $(m-n+1)(n+1)$ for the case of hypersurfaces $D_1,\ldots, D_q$ located in $m-$subgeneral position. For the case of arbitrary family of hypersurfaces, we do not know any result on that subject.

The last aim of this paper is to study singular directions of holomorphic curves. The existence of Julia directions was proved by G. Julia \cite{Ju} in 1920 for all entire
function and by Milloux in 1924 and Valiron in 1938 for most of meromorphic
functions. For general holomorphic curves from $\C$ to $\P^N(\C)$, it has been studied by Eremenko \cite{Ere}, Zheng\cite{Zh},...In \cite{Ere}, Eremenko showed the existence of Julia directions for some special classes of holomorphic curves into projective space and gave a definition of general Julia
directions for the “subvariety” case.
\begin{Def}\cite[p.4]{Ere}
A number $\theta \in [0;2\pi)$ is called a Julia direction for a holomorphic curve $f: \C\rightarrow V$ 
if for every system of divisors $D_1,\ldots, D_q$ such that any $n+1$ of them has empty
intersection, and for any $\pi>\epsilon>0$ all but at most $2n$ of these divisors have infinitely
many preimages in the angle $\{z:|arg z-\theta|<\epsilon\}$	.
\end{Def}
Recently, using the ideas of Eremenko \cite{Ere}, Zheng \cite[p.867]{Zh}, Thai and Mai \cite{TM} gave the notion on singular directions for a holomorphic curve into $\P^N(\C)$ intersecting hypersurfaces in general position.
\begin{Def}\cite[Def 3.1, p. 867]{Zh} A ray $J(\theta)=\{z\in\C:\arg z=\theta\}$ is a $T-$ direction for a holomorphic curve $f:\C\rightarrow\P^N(\C)$ if for any $\epsilon (0<\epsilon <\pi)$, we have
	$$\limsup\limits_{r\rightarrow\infty}\dfrac{N_f(r, \Omega, H)}{T_f(r)}=0$$
	for at most $2N$ hyperplanes $H$ in general position in $\P^N(\C).$
\end{Def}
	\begin{Def} \cite{TM}Let $V$ be a complex projective variety in $\P^N(\C)$ of dimension $n\geq 1.$ Let $f: \C\rightarrow V$ be a holomorphic curve. A ray $J(\theta)$ is called a $\bar{T}-direction$ for $f$ if for any $\epsilon (0<\epsilon <\pi)$, we have
		$$\limsup_{r\rightarrow\infty}\dfrac{N_f(r,\Omega, D)}{T_f(r)}=0$$
		for at most $n-1$ hypersurfaces $D$ located in general position in $V.$
		\end{Def}
Obviously, $\bar{T}-$direction is a $T-$ direction and they are Julia directions in the sense of Eremenko \cite{Ere}. However, this definition depends on the dimension of variety which is different from Eremenko's notion. We now extend Thai-Mai's notion on singular directions to not depend on the dimension of variety. Only the intersection pattern is relevant.
	\begin{Def}
Let $V$ be a complex projective variety in $\P^N(\C)$. Let $f: \C\rightarrow V$ be a holomorphic curve. A ray $J(\theta)$ is called a $\bar{T}-direction$ for $f$ if for any $\epsilon (0<\epsilon <\pi)$ and for any hypersurfaces $D_1,\ldots, D_q$ in $\P^N(\C)$ satisfying $\bigcap_{1\leq i\leq q}\supp D_i\bigcap V$ consisting of finite points,  we have
$$\max_{1\leq i\leq q}\limsup\limits_{r\rightarrow\infty}\dfrac{N_f(r,\Omega, D_i)}{T_f(r)}>0$$
\end{Def}
Based on Theorem \ref{t2}, we state the third main result of this paper.
\begin{The}\label{t3}
		Let $f: \C\rightarrow V$ be a holomorphic curve satisfying $\|f'(z)\|=O(|z|^{\sigma})$ for some $\sigma>-1$. If $\limsup\limits_{r\rightarrow\infty}\dfrac{T_f(r)}{r^{2\sigma+2}}>0$ then $f$ has at least one $\bar{T}-$direction.
	\end{The}

 We would like to remark that Theorem \ref{t2} and Theorem \ref{t3} imply Theorem \cite[Theorem 1.7]{TM} and Theorem \cite[Theorem 1.8]{TM}, respectively.
 \section{Some Lemmas}
Let $u$ be a subharmonic function in a domain $\Omega\subset \C.$ Let $\triangle u$	be the Riesz measure of $u$. We recall the following lemma due to Duval and Da Costa \cite[p.1599]{DD} reformulated by Thai-Mai \cite[Lemma 2.3]{TM}.
\begin{Lem}\cite[Lemma 2.3]{TM}(see also \cite[p.1599]{DD})\label{l2}
	Let $q, m$ be positive integers. Let $\nu_i, 1\leq i\leq q$ and $\nu$ be subharmonic functions in $\Omega.$ Assume that $\triangle\nu$ is $L^\infty$ on $\Omega$ and $\nu=\max_I\nu_i$ for any subset $I\subset\{1,\ldots,q\}, |I|=m.$ Then $\sum_{1\leq i\leq q}\nu_i-(q-m+1)\nu$ is subharmonic in $\Omega.$
\end{Lem}

	 Let $f: \C\rightarrow V$ be a holomorphic curve. We choose a homogeneous representation  ${\bf f}=(f_0,\ldots, f_N)$ of our curve.
	
	Let $D_1,\ldots, D_q$ be hypersurfaces of degrees $d_1,\ldots, d_q$ in $\P^N(\C)$, defining by homogeneous polynomials $Q_1,\ldots, Q_q\in\C[X_0,\ldots, X_N]$, located in $m-$subgeneral position in $V.$ Replacing $Q_i$ by $Q_i^{d/d_i}$ if necessary, where $d$ is the l.c.m of $d_i, 1\leq i\leq q,$ without loss of generality, we can assume that $Q_i, 1\leq i\leq q$ have the same degree $d$. Define
	$$ u=\log \|{\bf f}\|=\log\sqrt{|f_0|^2+\ldots+ |f_N|^2},$$
	 $$u_i=\dfrac{1}{d}\log |Q_i(f_0,\ldots, f_N)|, 1\leq i\leq q.$$
	 We will prove the following.
	 \begin{Lem}\label{l1}Let $f: \C\rightarrow V$ be a holomorphic curve satisfying $\|f'(z)\|=O(|z|^\sigma)$ for some $\sigma>-1.$ Let $r_k$ be a sequence of positive real numbers which converges to $+\infty.$ Let $h_k$ be the smallest harmonic majorant of $u$ in the disk $2r_k\mathbb{D}$ and $\lambda_k: \C\rightarrow \C$ be the map given by $\lambda_k(z)=r_kz.$ Then, there exist subsequences of the sequences $\frac{1}{r_k^{2\sigma+2}}(u_i-h_k)\circ \lambda_k$ and $\frac{1}{r_k^{2\sigma+2}}(u-h_k)\circ \lambda_k$ which converge to subharmonic functions $\nu_i$ and $\nu$, respectively, and  $\nu=\max_I\nu_i$ for any subset $I\subset\{1,\ldots,q\}$ such that $|I|=m.$
	 \end{Lem}
 \begin{proof} For simplicity, we can assume that $I=\{1,\ldots,m\}$.
 	If $q\geq m+1$, from the fact that
 	 since $D_i, 1\leq i\leq q$ are located in $m-$subgeneral position, we have $Q_i, 1\leq i\leq m+1$ have no common zeros in $V$. In the case of $q=m,$ we can choose a hypersuface $D_{m+1}$ of $\P^N(\C)$ defining by a homogeneous polynomial $Q_{m+1}$ such that $Q_i, 1\leq i\leq m+1$ have no common zeros in $V$.  Therefore, there exist two  positive constants $C_1, C_2 $ such that 
 $$C_1\leq \max_{1\leq i\leq m+1}|Q_i(\omega)|^{1/d}\leq C_2, \forall \omega\in \pi^{-1}(V)\bigcap\{\|z\|\leq 1\},$$ where  $\pi: \C^{N+1}\backslash\{0\}\mapsto \P^N(\C)$  be the standard projection. 
 Hence, 
 \begin{align}\label{e2001}
 	 C_1\|{\bf f}(z)\|\leq \max_{1\leq i\leq m+1}|Q_i\circ{\bf f}(z)|^{1/d}\leq C_2\|{\bf f}(z)\|, \forall z\in \C.
 \end{align}
It follows that 
 \begin{align}\label{e2}
  u=\max_{1\leq i\leq m+1}u_i+O(1)
 \end{align}
Consider the sequence of non-positive subharmonic functions on $2\mathbb{D}$ defined by $\dfrac{1}{r_k^{2\sigma+2}}(u-h_k)\circ\lambda_k.$ Note that $h_k(0)=\dfrac{1}{2\pi}\int_0^{2\pi} u(2r_ke^{i\theta})d\theta\leq O(r_k^{2\sigma+2})$ because $T_f(r)=O(r_k^{2\sigma+2}).$ Hence 
$\liminf\limits_{k\rightarrow \infty}\dfrac{1}{r_k^{2\sigma+2}}(u(0)-h_k(0))>-\infty$. So the sequence $\dfrac{1}{r_k^{2\sigma+2}}(u-h_k)\circ\lambda_k$ does not converge locally uniformly to $-\infty.$ By \cite[Theorem 4.1.9]{Hor}, there exists a subsequence of the sequence $\dfrac{1}{r_k^{2\sigma+2}}(u-h_k)\circ\lambda_k$  (still denoted by $k$) such that
$$\dfrac{1}{r_k^{2\sigma+2}}(u-h_k)\circ\lambda_k\rightarrow \nu,$$
where $\nu$ is a subharmonic function in $2\mathbb{D}.$ 

 For each $i,$ from \eqref{e2}, we have $u_i-h_k\leq u-h_k+O(1)<O(1).$ Therefore, $\dfrac{1}{r_k^{2\sigma+2}}(u_i-h_k)\circ\lambda_k$ is bounded above. Since $Q_i\circ{\bf f}\not=0$, there exists $y\in 2\mathbb{D}$ such that $u_i\circ\lambda_k(y)\not=-\infty.$ For sufficiently large $k$, $u(2r_ke^{i\theta})>0$ for all $\theta$. By Poisson's formula, we have
 \begin{align*}
 	h_k\circ\lambda_k(y)&=\dfrac{1}{2\pi}\int_{0}^{2\pi}\dfrac{4-|y|^2}{|2e^{i\theta}-y|^2}u\circ\lambda_k(2e^{i\theta})d\theta\\
 	&\leq \dfrac{2+|y|}{2\pi(2-|y|)}\int_{0}^{2\pi}u(2r_ke^{i\theta})d\theta\leq O(r_k^{2\sigma+2}).
 \end{align*}
Therefore, $\liminf\limits_{k\rightarrow \infty}\dfrac{1}{r_k^{2\sigma+2}}(u_i-h_k)\circ\lambda_k(y)>-\infty$. Hence,
 there exists a subsequence of the sequence $\dfrac{1}{r_k^{2\sigma+2}}(u_i-h_k)\circ\lambda_k$  (still denoted by $k$) such that
$$\dfrac{1}{r_k^{2\sigma+2}}(u_i-h_k)\circ\lambda_k\rightarrow \nu_i,$$
where $\nu_i$ is a subharmonic function in $2\mathbb{D}.$
 From \eqref{e2}, we have 
 \begin{align}\label{e4}
 	\nu=\max_{1\leq i\leq m+1}\nu_i.
 \end{align}
We now prove that 
$$\nu\leq \max_{1\leq i\leq m}\nu_i.$$
Indeed,	let $Q: V\rightarrow \P^m(\C)$ be a morphism defined by $$z\mapsto Q(z)=(Q_{m+1}(z), Q_1(z)\ldots, Q_{m}(z)).$$
Since $D_i, 1\leq i\leq q$ are located in $m-$subgeneral position, $Q_i, 1\leq i\leq m+1$ have no common zeros in $V$. Hence, $Q$ is a holomorphic mapping on $V$. Combining with the fact that $V$ is a compact complex projective variety and  $\|f'(z)\|=O(|z|^\sigma),$ we have $$g=(g_0,\ldots, g_{m}):=Q\circ {\bf f}: \C\rightarrow \P^m(\C)$$ satifying $\|g'(z)\|=O(|z|^\sigma)$.
By composing $g$ with an automorphism of $\P^m(\C)$, for example replace $g_0$ by $g_0+cg_1, c\in \C$ and leave all other $g_i$ unchanged, we can assume that $g_0$ has infinitely many zeros. Then, we can apply \cite[Proposition 2]{BE} due to Barret and Eremenko \cite{BE} to $g$. Then, for every $\epsilon >0,$ we have for $z$ sufficiently large,
\begin{align}
	\log\sqrt{|g_0|^2+\ldots+|g_m|^2}(z)\leq \max_{1\leq i\leq m}\log |g_i|(z)+K(2+\epsilon)^{\sigma+1}(m+1)|z|^{\sigma+1}.
\end{align}
where $K$ is a constant depending only on $f$ and $Q.$
Hence,
\begin{align}\label{e1}
	\max_{1\leq i\leq m+1}\log |Q_i\circ {\bf f}(z)|\leq d\max_{1\leq i\leq m}u_i(z)+K(2+\epsilon)^{\sigma+1}(m+1)|z|^{\sigma+1}.
\end{align}
 From \eqref{e1} and \eqref{e2001}, we have for $z$ sufficiently large
\begin{align}\label{e3}
	u(z)\leq \max_{1\leq i\leq m}u_i(z)+K(2+\epsilon)^{\sigma+1}(m+1)|z|^{\sigma+1}.
\end{align}
  From \eqref{e3}, we have
 \begin{align}\label{e5}
 	\nu\leq \max_{1\leq i\leq m}\nu_i.
 \end{align}
From \eqref{e4} and \eqref{e5}, we have $\nu=\max_{1\leq i\leq m}\nu_i.$
 \end{proof}
\begin{Lem} \label{l3}
$\sum_{1\leq i\leq q}\nu_i-(q-m+1)\nu$ is subharmonic in $2\mathbb{D}.$
\end{Lem}
\begin{proof}
	In view of Lemma \ref{l1} and Lemma \ref{l2}, it suffices to prove that	the Riesz measure $\triangle\nu$ is $L^{\infty}$ on $2\mathbb{D}.$
	
	Indeed, since $\dfrac{1}{r_k^{2\sigma+2}}(u-h_k)\circ\lambda_k\rightarrow \nu$, we have
	 $\dfrac{1}{r_k^{2\sigma+2}}\int \phi\triangle(u\circ\lambda_k )\rightarrow \int\phi\triangle\nu$ for any test function $\phi\in C_0^\infty(2\mathbb{D}).$
	 
We denote by $B(z,r)$ the open disc of radius $r$ centered at $z$. Hence, for any $B(z,\delta)\subset 2\mathbb{D},$ we have
	$$\triangle\nu(B(z,\delta))\leq \liminf_{r\rightarrow\infty}\dfrac{\triangle u(B(rz,r\delta))}{r^{2\sigma+2}}\leq \liminf_{r\rightarrow\infty}\dfrac{\delta^{2\sigma+2}r^{2\sigma+2}O(1)}{r^{2\sigma+2}}=O(1)\delta^{2\sigma+2},  $$
	where $O(1)$ is a constant depending only on $f.$ Hence $\triangle\nu$ is $L^\infty$ on $2\mathbb{D}.$
\end{proof}

\section{Proof of main theorems}
In this section, we provide the proofs of main theorems. More precisely, Theorem \ref{t1} follows from \cite{BE} and inequality \ref{e2}. Theorem \ref{t2} is obtained from Lemma \ref{l3} and Theorem \ref{t3} is a consequence of Theorem \ref{t2}. The details are given as follows.\\
{\it  Proof of Theorem \ref{t2}. }

We first prove that
\begin{align}\label{e21}
	\limsup_{r\rightarrow\infty} \frac{1}{r^{2\sigma+2}}\left[(q-m+1)\triangle u(\mathbb{D}_r)-\sum_{1\leq i\leq q} \triangle u_i (\mathbb{D}_r)\right]\leq 0,
\end{align}
where $\mathbb{D}_r=\Omega \bigcap\{|z|<r\}.$
Indeed, assume that $r_k\rightarrow +\infty$ such that 
\begin{align*}
	\frac{1}{r_k^{2\sigma+2}}\left[(q-m+1)\triangle u(\mathbb{D}_{r_k})-\sum_{1\leq i\leq q} \triangle u_i (\mathbb{D}_{r_k})\right]	
\end{align*}
converges.

Applying Lemma \ref{l3}, we have $\sum_{1\leq i\leq q}\nu_i-(q-m+1)\nu$ is subharmonic in $2\mathbb{D}\bigcap \Omega$. Fix $\delta>0$ with $0<\delta<\min\{1/2,\epsilon\}$. We choose a nonnegative smooth function $\chi$ in $\mathbb{D}\bigcap \Omega$ such that $0\leq \chi\leq 1$ and $\chi=1$ in $\mathbb{D}^\delta=(1-\delta)\mathbb{D}\bigcap \Omega(\theta, \epsilon-\delta)\bigcap\{|z|>\delta\}$ and $\supp\chi$ is a compact subset in $\mathbb{D}\bigcap\Omega.$ Since $\sum_{1\leq i\leq q}\nu_i-(q-m+1)\nu$ is subharmonic, we have
$$(q-m+1)\int\chi\triangle\nu\leq \sum_{1\leq i\leq q}\int \chi\triangle \nu_i,$$
where the integration is taken over $\mathbb{D}\bigcap \Omega.$ It follows that for $k$ sufficient large, we get
\begin{align*}
	(q-m+1)\dfrac{1}{ r_k^{2\sigma+2}}\int\chi\triangle(u\circ\lambda_k)\leq \sum_{1\leq i\leq q}\dfrac{1}{ r_k^{2\sigma+2}}\int\chi\triangle(u_i\circ\lambda_k)+\delta.
\end{align*}
Therefore, 
\begin{align}\label{e31}
	(q-m+1)\triangle u(r_k\mathbb{D}^\delta)\leq \dfrac{1}{d}\sum_{1\leq i\leq q} \triangle u_i(r_k\mathbb{D}\bigcap\Omega)+\delta r_k^{2\sigma+2},
\end{align}
where $r_k\mathbb{D}^\delta=(1-\delta)r_k\mathbb{D}\bigcap \Omega(\theta, \epsilon-\delta)\bigcap\{|z|>\delta r_k\}$. Furthermore, since $\|f'(z)\|=O(|z|^\sigma)$, we have 
\begin{align}\label{e32}
	\triangle u(r_k\mathbb{D}\bigcap\Omega)\leq \triangle u(r_k\mathbb{D}^\delta)+(1+3\epsilon)O(1)\delta( r_k)^{2\sigma+2}	.
\end{align}
From \eqref{e31} and \eqref{e32} and letting $\delta\rightarrow 0$, we have
\begin{align*}
	\lim_{k\rightarrow}\dfrac{1}{r_k^{2\sigma+2}}\left[(q-m+1)\triangle u(\mathbb{D}_{r_k})-\sum_{1\leq i\leq q}\triangle u_i(\mathbb{D}_{r_k})\right]\leq 0.
\end{align*} 
Hence, \eqref{e21} holds. From \eqref{e21}, we have
\begin{align*}
	(q-m+1)\triangle u(r\mathbb{D}\bigcap\Omega)&\leq \sum_{1\leq i\leq q} \triangle u_i(r\mathbb{D}\bigcap\Omega)+o(r^{2\sigma+2})\\
	&=n_f(r, r\mathbb{D}\bigcap\Omega, D_i)+o(r^{2\sigma+2}),
\end{align*}
where $n_f(r, r\mathbb{D}\bigcap\Omega, D_i)$ is the number of zeros of $Q_i\circ{\bf f}$ in the domain $r\mathbb{D}\bigcap \Omega$, counting multiplicity.
Dividing both sides of above inequality to $r$ and integrating it, we get the conclusion.\\
	{\it Proof of Theorem \ref{t3}.}
	Suppose that $f$ has no $\bar{T}-$direction. Then, for each ray $J(\theta)$, we have an angle containing it such that 
	$$\limsup_{r\rightarrow\infty} \dfrac{N_f(r, \Omega(\theta,\epsilon),D_i)}{T_f(r)}=0$$
for hypersurfaces $D_1,\ldots, D_q$ in $\P^N(\C)$ satisfying $\bigcap_{1\leq i\leq q}\supp D_i\bigcap V$ consisting of finite points. Hence, we can choose finite family of angles $\Omega(\theta_i, \epsilon_i) (1\leq i\leq p)$ such that 
$$\C\subset\cup_{1\leq i\leq p}\Omega(\theta_i, \epsilon_i) )$$
 and 
\begin{align}\label{e205}
	\limsup_{r\rightarrow\infty}\dfrac{N_f(r, \Omega(\theta_i,\epsilon_i),D_i)}{T_f(r)}=0.
\end{align}
	for hypersurfaces $D_1,\ldots, D_q$ in $\P^N(\C)$ satisfying $\bigcap_{1\leq i\leq q}\supp D_i\bigcap V$ consisting of finite points.
Therefore,
\begin{align}\label{e202}
	T_f(r)\leq \sum_{1\leq i\leq p}T_f(r, \Omega(\theta_i, \epsilon_i) ).
\end{align}
 Applying Theorem \ref{t2} to $D_i, 1\leq i\leq q$, we have
\begin{align}\label{e203}
	\sum_{1\leq k\leq p}T_f(r, \Omega(\theta_k, \epsilon_k) )\leq \sum_{1\leq k\leq p}\sum_{1\leq i\leq q}\dfrac{1}{d_i}N_f(r, \Omega(\theta_k, \epsilon_k),D_i)+o(r^{2\sigma+2}).
\end{align}
From \eqref{e202} and \eqref{e203}, we have
\begin{align}\label{e204}
	T_f(r )\leq \sum_{1\leq k\leq p}\sum_{1\leq i\leq q}\dfrac{1}{d_i}N_f(r, \Omega(\theta_k, \epsilon_k),D_i)+o(r^{2\sigma+2}).
\end{align}
From \eqref{e205} and \eqref{e204} and the fact that $\limsup\limits_{r\rightarrow\infty}\dfrac{T_f(r)}{r^{2\sigma+2}}>0$, we get a contradiction.\\
{\it Proof of Theorem \ref{t1}.}
Without loss of generality, we can assume that $D_1,\ldots, D_q$ have the same degree $d$. Let $Q_1,\ldots, Q_q$ be the homogeneous polynomials of degree $d$ defining $D_1,\ldots, D_q$.
Since $\bigcap_{1\leq i\leq q}\{Q_i=0\}\bigcap V$ consists of finite points, there exists a hypersurface $D_0$ of $\mathbb{P}^N(\mathbb{C})$ defining by a homogeneous polynomial $Q_0\in \mathbb{C}[X_0,\ldots, X_N]$ such that $\bigcap_{0\leq i\leq q}\{Q_i=0\}\bigcap V=\emptyset.$ Consider the map
\begin{align*}
	Q=(Q_0,\ldots, Q_q): V\rightarrow \P^q(\C), z\mapsto (Q_0(z),\ldots, Q_q(z)).
\end{align*}
This map is a holomorphic map between two compact complex varieties. Together with the fact that $\|f'(z)\|=O(|z|^\sigma)$, we have $Q\circ f: \mathbb{C}\rightarrow \P^q(\C)$ satisfies $\|Q'\circ f(z)\|=O(|z|^\sigma)$. Since $\bigcap_{1\leq i\leq q}\{Q_i=0\}\bigcap f(\C)=\emptyset,$ $Q\circ f$ omits $q$ hyperplanes in general position. Hence, the main theorem in \cite{BE} is applicable to $Q\circ f$. Then, we have
$T_{Q\circ f}(r)=O(r^{\sigma+1}).$ Using the same arguments as in Lemma \ref{l1} to obtain inequality \ref{e2}, we have $T_{Q\circ f}(r)=dT_f(r)$. This completes the proof of Theorem \ref{t1}.

\section{Acknowledgements}
This work was supported by a NAFOSTED, grant of Vietnam (Grant No. 101.02-2021.16).
	
\end{document}